\documentclass[12pt]{amsart}
\usepackage{}

\usepackage{amsfonts}

\usepackage[alphabetic]{amsrefs}
\usepackage{amscd}
\usepackage{amsmath}
\usepackage{amssymb}
\usepackage{amsthm}

\usepackage{bigdelim}
\usepackage{color}
\usepackage{enumerate}
\usepackage{mathrsfs}
\usepackage{multirow}
\usepackage[all]{xy}
\newtheorem{thm}{Theorem}[section]
\newtheorem{cor}[thm]{Corollary}

\newtheorem{prop}[thm]{Proposition}
\newtheorem{conj}[thm]{Conjecture}

\newtheorem{lem}[thm]{Lemma}
\theoremstyle{definition}

\theoremstyle{remark}
\newtheorem{rmk}[thm]{Remark}

\newtheorem{ques}[thm]{Question}

\newcommand{\sE}{\mathscr{E}}

\newcommand{\cO}{\mathcal{O}}

\newcommand{\cL}{\mathcal{L}}
\newcommand{\cN}{\mathcal{N}}
\newcommand{\cM}{\mathcal{M}}
\newcommand{\sk}{\mathbf{k}}

\usepackage{titletoc}

\begin{document}
\title[Counterexamples to Fujita's conjecture on surfaces in char. $p>0$]{Counterexamples to Fujita's conjecture on surfaces in positive characteristic}
\author[Yi Gu]{Yi Gu}
\address{School of Mathematical Sciences, Soochow University, Suzhou, 215006, P. R. of China}
\email{sudaguyi2017@suda.edu.cn}
\author[Lei Zhang]{Lei Zhang}
\address{School of Mathematical Sciences,
	University of Science and Technology of China, Hefei, 230026, P. R. of China}
\email{zhlei18@ustc.edu.cn}
\author[Yongming Zhang]{Yongming Zhang}
\address{ School of Mathematics, China University of Mining and Technology, Xuzhou, 221116, P. R. of China}
\email{zhangym@cumt.edu.cn}

\begin{abstract}
We present counterexamples to Fujita's conjecture in positive characteristic. More precisely,  given any algebraically closed field $\sk$ of characteristic $p>0$ and any  positive integer $m$, we show there exists a smooth projective surface $S$ over $\sk$ admitting an ample Cartier divisor $A$ such that the adjoint linear system $|K_S+mA|$ is not free of base points.
\end{abstract}
\maketitle

\section{Introduction}
The study of adjoint linear systems plays an important role in the classification of varieties.  In this direction,  T. Fujita \cite{Fu85} proposed the following famous conjecture:
\begin{conj}[Fujita's conjecture]\label{Fujita-conj}
Let $X$ be a smooth projective variety of dimension $n$ over an algebraically closed field of characteristic zero and $A$ an ample Cartier divisor on $X$.
\begin{enumerate}
\item[{$\mathrm{i)}$}]  $\mathrm{(Freeness)}$ For $m\geq n+1$, the adjoint linear system $|K_X + mA|$ is free of base points.
\item[{$\mathrm{ii)}$}] $\mathrm{(Very ~~ampleness)}$ For $m\geq n+2$, the adjoint linear system $|K_X + mA|$ is very ample.
\end{enumerate}
\end{conj}
\noindent By taking $X=\mathbb{P}^n$ and $A$ a hyperplane, it can be seen that the bounds of $m$ in both the freeness and the very ampleness parts of this conjecture are optimal.

On curves Fujita's conjecture follows from Riemann-Roch formula immediately  and on surfaces it has been proved  by Reider's elegant method \cite{Re88}. The freeness part of this conjecture has now been proved up to dimension five (see \cite{EL93, He97, Ka97, YZ20})  while the very ampleness part remains widely open when $\dim X \geq 3$.  On the other hand, for varieties of arbitrary dimension,
there have been many other important ``Fujita's conjecture type'' results (see \cites{Dem93, Kol93, AS95, He97, Sm97} and so on). One of these remarkable results was  due to Angehrn and Siu \cite{AS95}: $|K_X + mA|$ is base point free for $m\geq \begin{pmatrix}n+1\\ 2\end{pmatrix} +1$. We refer the readers to \cite[\S~10.4A]{PosII} for a brief review about both the related results and techniques.
\smallskip

Although Fujita's conjecture was originally formulated in characteristic zero, many people have studied the positive characteristic version. For convenience, let us simply say Conjecture~\ref{Fujita-conj} to mean `the positive characteristic version of Conjecture~\ref{Fujita-conj}' when there is no risk of confusion.

Along this direction, in dimension two, by adapting Reider's method to characteristic $p$, Shepherd-Barron \cite[Corollary 8]{SB91} was able to prove Conjecture~\ref{Fujita-conj} in dimension two except for surfaces of general type and  quasi-elliptic  surfaces, and more recently it is claimed that Conjecture~\ref{Fujita-conj} also holds for the latter case by Chen \cite{Ch19}.  On the other hand, following from the celebrated works of Deligne and Illusie \cite{DI87} and of Langer \cite{Lan15}, Reider's method  applies to surfaces admitting a $W_2$-lifting and therefore Conjecture~\ref{Fujita-conj} is also true in the $W_2$-lifting case. Besides, some other `Fujita's conjecture type' results are also obtained for surfaces (see \cite{Na93, Na93b, Ter99, DCF15} and so on). For example Di Cerbo and Fanelli \cite{DCF15} proved $|2K_X + 38A|$ is very ample for a surface.

For varieties of arbitrary dimension, by results of Smith and Keeler (\cite{Ke08, Ke08E, Sm97, Sm00}), it is known that Conjecture \ref{Fujita-conj}  is also true under an additional assumption that $\mathcal{O}_X(A)$ is globally generated. Furthermore, there are many other attempts to prove Conjecture~\ref{Fujita-conj} in positive characteristic  (see \cite{Sch14,MS14, Ma19}  and so on). However, we find examples disproving Fujita's conjecture in positive characteristic, and there does not even exist a \emph{Fujita type bound}. In particular,
we have a negative answer to Schwede's question \cite[pp.~71]{Sch14} concerning the stable adjoint linear subsystems in dimension at least two.
Our main result is the following theorem.
\begin{thm}\label{thm-main}
 Let $\sk$ be an arbitrary algebraically closed field of positive characteristic and $m\in \mathbb{N}_+$ an arbitrary positive integer.  Then there exists a smooth projective surface $S$ over $\sk$ admitting an ample Cartier divisor $A$  such that the  complete linear system $|K_S + mA|$ is not free of base point.
\end{thm}

The surface $S$ in Theorem~\ref{thm-main} is a generalization  of the so-called Raynaud's surface  (\cite{Ra78, Mu13}).  As is known to experts, comparing with the case in characteristic zero, one disadvantage in dealing with Conjecture~\ref{Fujita-conj} is the failure of Kodaira type vanishing  in positive characteristic\footnote{However, Murayama's opinion that `the failure of Kodaira-type vanishing is not the main obstacle to Fujita's conjecture'  seems right \cite[Principal~1.6]{Ma19}.  For instance, as Chen \cite{Ch19} claimed, Fujita's conjecture holds for quasi-elliptic surfaces while Kodaira's vanishing fails on these surfaces \cite{Ra78,Zhe17}.}.  In history, Raynaud's surface is the first counterexample to Kodaira's vanishing (\cite{Ra78}). So we firstly  checked Conjecture~\ref{Fujita-conj} for Raynaud's surfaces,  and with a non-trivial observation  we actually discovered that Conjecture~\ref{Fujita-conj} already fails for some special Raynaud's surfaces provided the characteristic $p$ is large. To obtain examples for all primes, we made a `passage from $p$ to $q=p^n$' generalization of  Raynaud's surfaces to obtain the final example (cf. \S~\ref{Sec: a generalization of Raynaud's surface}) as required in Theorem~\ref{thm-main}.

\smallskip

Finally, concerning with the Seshadri constant, we have the following comments on  the pair $(S,A)$ in Theorem~\ref{thm-main}. Fixing the base field $\sk$ and $3\le m\in \mathbb{N}_+$, there is a divisor $\Gamma_2$ on $S$ (cf. \S~\ref{Subsec: generalized raynaud}) and a positive constant $b(m)$ such that the Seshadri constant  $\varepsilon(A,x)<b(m)$ for any $x\in \Gamma_2$ and $\varepsilon(A, x)\geq 1$ when $x\in S\setminus \Gamma_2$. The constants $b(m)$ are furthermore such that $\lim\limits_{m\to \infty}b(m)=0$.  In particular,  although the adjoint system $|K_S + mA|$  is not free of base points,  it does define a birational map whenever $m> 4$. Indeed, as $\varepsilon(A,x)\geq 1$ when $x\in S\setminus \Gamma_2$, we have the Frobenius-Seshadri constant $\varepsilon_F(mA,x)\ge \dfrac{m}{2}>2$ for such $x$ (cf. \cite[Prop.~2.12]{MS14}). It then follows from  \cite[Thm.~1.1(2)]{MS14} that $|K_S+mA|$ defines a birational map.  So in spite that effectivity on the freeness and very ampleness of  adjoint linear systems fails due to Theorem~\ref{thm-main}, the effectivity on birationality still makes sense.
 \begin{ques}
Is there a constant $B(n,p)$ such that for any  $n$-dimensional smooth projective variety $X$ defined over a field of characteristic $p$ and an ample Cartier divisor $A$ on $X$, the adjoint linear system $|K_X+mA|$ defines a birational map if $m\ge B(n,p)$?  Furthermore, can we expect $B(n,p)$ is independent on $p$?
\end{ques}

\medskip

\textbf{Acknowledgement:} We would like to thank  Sho Ejiri, Hiromu Tanaka, Takumi Murayama, Qingyuan Xue, Lifan Guan and Yifei Chen for their useful comments and suggestions. The first named author is supported by grant
NSFC (No. 11801391).  The second author is partially supported by supported by the National Key R and D Program of China (No. 2020YFA0713100) and NSFC (No. 11771260) and the Fundamental Research Funds for Central Universities. The corresponding author is supported by grant NSFC (No. 12101618).

\section{A generalization of Raynaud's surface}\label{Sec: a generalization of Raynaud's surface}
In this section, we are going to present a certain generalization of Raynaud's surface  given in \cite{Ra78}. In the next section, we will show our generalization does give the example as required in Theorem~\ref{thm-main}.  Let us fix an algebraically closed field $\sk$ of characteristic $p>0$.

\subsection{Outline of the generalization}\label{Subsec: outline}
Let us briefly recall Raynaud's surface and then outline how we generalize his construction.

Raynaud's surface $S_1$ is, roughly speaking, a cyclic cover $\pi_1: S_1\to \mathbb{P}$ of degree prime to $p$ over a ruled surface  $\rho: \mathbb{P}\to C$ over a curve $C$ such that $\pi_1$ is branched along two \emph{smooth} irreducible divisors $\Sigma_i\subseteq \mathbb{P},i=1,2$ satisfying
\begin{enumerate}[(i)]
\item    $\rho|_{\Sigma_1}: \Sigma_1\to C$ is an isomorphism, namely, $\Sigma_1$ is a section of $\rho$;

\item    $\rho|_{\Sigma_2}: \Sigma_2\to C$ is  a purely inseparable morphism of degree $p$; and

\item $\Sigma_1\cdot \Sigma_2=0$.
\end{enumerate}

Our generalization of Raynaud's construction is to replace (ii)  with
\begin{enumerate}[(ii*)]
\item the restriction map $\rho|_{\Sigma_2}: \Sigma_2\to C$ is a purely inseparable morphism of degree $q=p^n$ for  any prescribed $n\in \mathbb{N}_+$.
\end{enumerate}

\subsection{The base curve}
Fix an arbitrary positive integer $n$ and set $q=p^n$.  We  firstly  give the base curve $C$, which is obtained by a slight modification of the example in \cite[Exm.~1.3]{Mu13}.

Let $C\subseteq \mathbb{P}_\sk^2=\mathrm{Proj}(\sk[X,Y,Z])$ be the plane curve defined by the equation:
\begin{equation}\label{defining equation}
Y^{qe}-X^{qe-1}Y=XZ^{qe-1},
\end{equation}where $e\in \mathbb{N}_+$ is a free variable. It is easy to   check  that $C$ is a smooth curve and  $$2g(C)-2=qe(qe-3).$$

Take $\infty:=[0,0,1]$ on $C$. Then $U_1:=C\backslash \infty=C\cap \{ X\neq 0\}$ is an affine open subset defined by $y_1^{qe}-y_1=z^{qe-1}$ with $y_1=Y/X$ and $z=Z/X$. As a result, $\mathrm{d}z$ is a generator of $\Omega^1_{C}|_{U_1}$ since $\mathrm{d}y_1=z^{qe-2}\mathrm{d}z$. In particular, we have
\begin{equation}\label{equ: formula of K_C}
K_C=\mathrm{div}(\mathrm{d}z)=(2g(C)-2)\infty=qe(qe-3)\cdot \infty.
\end{equation}

Next set $V:=C\cap \{Z\neq 0\}$.  Then $V\subset C$ is an open affine subset defined by the equation $y^{qe}-x^{qe-1}y=x$ with $y=Y/Z$ and $x=X/Z$. The special point $\infty$ is now given by  $x=y=0$.  After some easy local calculations, we have:
\begin{itemize}
\item $y$ is a local parameter at $\infty=[0,0,1]$ and invertible on $V\backslash \infty$, and

\item $v_\infty(x)=qe$. Indeed, by $x=y^{qe}-x^{qe-1}y$, it follows that
$$v_\infty(x) \geq \min\{qe\cdot v_\infty(y)=qe, (qe-1)\cdot v_\infty(x) + 1\} >0$$
and thus $(qe-1)\cdot v_\infty(x) + 1 > qe$, which gives the inequality $v_\infty(x)=qe$.
\end{itemize}
As a result, the function  $$\gamma:=\left(\dfrac{x}{y^{qe}}\right)^{qe-2}\cdot y=(1-\left(\dfrac{x}{y}\right)^{qe-1})^{qe-2}\cdot y=(1-(y^{qe-1}-x^{qe-1})^{qe-1})^{qe-2}\cdot y \in \cO(V)$$ is also a local parameter at $\infty$. In particular, $\mathrm{d}\gamma$ is a generator of $\Omega_{C,\infty}^1$ and hence there is an open neighbourhood $\infty\in U_2\subseteq V$ such that  $\mathrm{d}\gamma$ is a generator of  $\Omega_{C}^1|_{U_2}$.

Finally, we have
$$z-y^{-qe}=\dfrac{1}{x}-\dfrac{1}{y^{qe}}=\dfrac{x^{qe-1}y}{xy^{qe}}=y^{qe(qe-3)}\gamma,$$
or equivalently,
$$z=(y^{e(qe-3)})^q\gamma+(y^{-e})^q.$$

In summation, we have an affine covering $C=U_1\bigcup U_2$ along with two rational functions $z_1=z,z_2=\gamma$ such that
\begin{itemize}
\item $z_i$ is regular on $U_i$ and $\mathrm{d}z_i$ is a generator of $\Omega_{C/\sk}^1|_{U_i}$, and
\item the translation relation of $z_i$ is given by:
\begin{equation}\label{equ: translation relation}
z_1=\alpha^qz_2+\beta^q,
\end{equation}
where  $\alpha=y^{e(qe-3)}\in \cO_{C}(U_1\cap U_2)^*~\mathrm{and}~\beta=y^{-e}\in \cO_{C}(U_1\cap U_2).$
\end{itemize}
\begin{rmk}
Note that as $y$ and $\alpha$ are invertible on $V\backslash \infty$, it turns out that $\mathrm{d}\gamma=\mathrm{d}z_2$ is a generator of $\Omega_{C}^1|_V$ and we can in fact take $U_2=V$.
\end{rmk}

\subsection{A rank two locally free sheaf $\sE$ on $C$}\label{sec:vectorbundle}
By the previous constructions, $z_i\in \cO_{U_i}$ and $\mathrm{d}z_i$ is a generator of $\Omega_{C/\sk}^1|_{U_i}$. As a consequence,   in the following finite purely inseparable cover $V_i:=\mathrm{Spec}(\cO_{U_i}[t]/(t^q-z_i)) \to U_i$, the source $V_i$ is regular. In fact, by Jacobian criterion, $$\Omega_{V_i/\sk}=\dfrac{(\Omega_{U_i}\otimes_{\cO_{U_i}}\cO_{V_i})\oplus \cO_{V_i}\cdot\mathrm{d}t}{\cO_{V_i}\cdot\mathrm{d}(t^q-z_i)}=\dfrac{\cO_{V_i}\cdot\mathrm{d}z_i\oplus \cO_{V_i}\cdot\mathrm{d}t}{\cO_{V_i}\cdot\mathrm{d}z_i}\simeq \cO_{V_i}\cdot \mathrm{d}t\simeq \cO_{V_i}.$$
Therefore, the covering map $V_i\to U_i$ coincides with the $n$-th Frobenius map,  and  we have $$F^n_*\cO_C|_{U_i}=\cO_{U_i}[t]/(t^q-z_i)=\cO_{U_i}[\sqrt[q]{z_i}]=\bigoplus\limits_{j=0}^{q-1}\cO_C\cdot \sqrt[q]{z_i^j}.$$ Then we construct a rank two locally free subsheaf $\sE\subseteq F^n_*\cO_C$ as follows.
On $U_i$, $\sE$ is given by:
\begin{equation}\label{equ: Sym-equ}
\sE|_{U_i}=\cO_{U_i}\cdot X_i(:=1)\oplus \cO_{U_i} \cdot Y_i(:=\sqrt[q]{z_i}) \subseteq F^n_*\cO_C|_{U_i}=\bigoplus\limits_{j=0}^{q-1}\cO_C\cdot \sqrt[q]{z_i^j}.
\end{equation}
\noindent Here note that by equation (\ref{equ: translation relation}), we have the following translation relation on $U_1\cap U_2$:
$$
\left\{\begin{array}{rl}
X_1&=X_2,\\
Y_1&=\alpha\cdot Y_2+\beta \cdot X_2.
\end{array}\right.
$$
where $\alpha=y^{e(qe-3)}\in \cO_{C}(U_1\cap U_2)^*~\mathrm{and}~\beta=y^{-e}\in \cO_{C}(U_1\cap U_2).$ So the above construction of $\sE$ makes sense.

By construction, as a subsheaf  of $F^n_*\cO_C$, $\sE$ contains the natural saturated subsheaf $\cO_C\subseteq F^n_*\cO_C$  and hence   it induces an exact sequence:
\begin{equation}\label{ext: ext of E}
0\to \cO_C\to \sE\stackrel{\theta}{\to} \cL\to 0.
\end{equation}
As $\cO_C$ is locally generated by $1=X_i$ on $U_i$, $\cL$ is an invertible sheaf with generators $\eta_i:=\theta(Y_i)$ on $U_i$ and translation relation $\eta_1=\alpha\cdot \eta_2$. As $v_\infty(\alpha)=e(qe-3)$ and $\mathrm{div}(\eta_1)=e(qe-3)\cdot \infty$, we have
\begin{equation}\label{equ: formula of L}
\cL\cong\cO_C(e(qe-3)\cdot\infty) \,\, \text{and} \,\, \cL^q\cong \omega_{C}.
\end{equation}

\subsection{The ruled surface $\mathbb{P}(\sE)$}\label{subsec: ruled surface}
In this subsection we will study the ruled surface $$\rho: \mathbb{P}(\sE):=\mathrm{Proj}_{\cO_C}(\mathrm{Sym}(\sE))\to C.$$  As mentioned in \S~\ref{Subsec: outline}, this ruled surface   should  admit a special configuration of two   smooth irreducible divisors $\Sigma_1,\Sigma_2$. We now give their constructions in the following.

Firstly, the exact sequence (\ref{ext: ext of E}) gives a section divisor $\Sigma_1\in |\cO_{\mathbb{P}(\sE)}(1)|$, which  is locally defined by $X_i=0$ on $\mathbb{P}_{U_i}(\sE) (=\mathrm{Proj}(\cO_{U_i}[X_i,Y_i])) $. Then
\begin{itemize}
\item $\Sigma_1$ is a section of $\rho$; and

\item $\cO_{\Sigma_1}(1)=\rho^*\cL|_{\Sigma_1}, \Sigma_1^2=\deg \cL=(2g(C)-2)/q >0$ . In particular, $\Sigma_1$ is nef and big.
\end{itemize}

Secondly, the restriction of the multiplicative map to $\cO_C\otimes_{\cO_C} F^n_* \cO_C$  $$\iota: \sE\otimes_{\cO_C} F^n_* \cO_C \hookrightarrow F^n_*\cO_C\otimes_{\cO_C} F^n_*\cO_C \xrightarrow{multiplication} F^n_*\cO_C$$
gives a splitting of (\ref{ext: ext of E}) after tensoring with $F^n_*\cO_C$ as the left diagram below. This splitting then gives a map $\nu$ as the right commutative diagram below such that $\nu^*\cO_{\mathbb{P}(\sE)}(1)\simeq \cO_C$.
$$
\xymatrix{
\cO_C \otimes F^n_*\cO_C \ar[d]^{\simeq}\ar@{^{(}->}[rr] && \sE\otimes F^n_* \cO_C \ar[d]^\iota \\
F^n_*\cO_C      \ar@{=}[rr] && F^n_*\cO_C
}
\,\,\, \, \xymatrix{
C\ar[rrd]_{F^n} \ar[rrrr]^\nu &&&& \mathbb{P}(\sE) \ar[lld]^\rho\\
&& C
}
$$
In fact, this $\iota$ is nothing but a surjective map $(F^n)^*\sE\to \cO_C$ (here $C$ is the source of the $n$-th Frobenius map). Thus $\iota$ induces a morphism $\nu$ as above so that $\nu^*\cO_{\mathbb{P}(\sE)}(1)\simeq \cO_C$. We simply take $\Sigma_2:=\nu(C)$ and it follows from
$$
0=\deg_C(\nu^*\cO_{\mathbb{P}(\sE)}(1))=\deg(\nu)\cdot (\Sigma_1\cdot \Sigma_2)
$$
that  $\Sigma_1\cap \Sigma_2=\emptyset$.
More concretely, on $U_i$, $\iota$ is given by $X_i\mapsto 1$ and $Y_i\mapsto \sqrt[q]{z_i}$. Correspondingly, on $\mathbb{P}_{U_i}(\sE) (=\mathrm{Proj}(\cO_{U_i}[X_i,Y_i]))$, $\Sigma_2$ is defined by $Y_i^q - z_iX_i^q=0$ while $\Sigma_1$ is defined by $X_i=0$. Moreover, from these equations we see that $\Sigma_1\cap \Sigma_2=\emptyset$, and the projective $\Sigma_2 \to C$ coincides with $F^n: C\to C$ (cf. \S \ref{sec:vectorbundle}).

In summation, we have
\begin{itemize}
\item $\Sigma_2$ is smooth and $\rho|_{\Sigma_2}: \Sigma_2\to C$ is the $n$-th iterated Frobenius map of $C$;
\item $\Sigma_1\cap \Sigma_2=\emptyset$; and
\item  $\Sigma_2\in |\cO_{\mathbb{P}(\sE)}(q)\otimes \rho^*\omega_{C}^{-1}|$. In fact, we have $\Sigma_{2}\in |\cO_{\mathbb{P}(\sE)}(q)\otimes \rho^*\mathcal{H}|$ for some invertible sheaf $\mathcal{H}$ on $C$. As $\Sigma_2\cap \Sigma_1=\emptyset$, we have $(\cO_{\mathbb{P}(\sE)}(q)\otimes \rho^*\mathcal{H})|_{\Sigma_1}=\cO_{\Sigma_1}$ and hence $\rho^*\mathcal{H}\simeq \cO_{\Sigma_1}(-q)\simeq \rho^*\cL^{-q}\simeq \rho^*\omega_C^{-1}$.
\end{itemize}

So far, we have given the configuration as promised in  \S~\ref{Subsec: outline}.

\subsection{A generalized Raynaud's surface}\label{Subsec: generalized raynaud}
Given the triple pair $(\mathbb{P}(\sE),\Sigma_1,\Sigma_2)$ as above, we shall construct a generalization of Raynaud's surface.

In the following we assume in the defining equation~(\ref{defining equation}) of $C$, the free integral variable $e$ is such that
\begin{enumerate}[($\blacksquare$)]
\item $\,\,\,\,\,\,\,\,\,\,\,\,\,\,\,\,\,\,\,\,\,\,\,\,\,\,\,\,\,\,\,\,\,\, (q+1)\mid qe(qe-3)=2g(C)-2$.
\end{enumerate}
This assumption can be easily fulfilled, {\it e.g.}, by taking $(q+1)\mid e$. Then denote the invertible sheaf   $\cN:=\cO_C(\dfrac{qe(qe-3)}{q+1}\cdot \infty)$, and we have $$\cN^{q+1}\simeq \cO_C(qe(qe-3)\cdot \infty)\simeq \omega_C$$ by (\ref{equ: formula of K_C}).
As a result, the invertible sheaf $\cM:=\cO_{\mathbb{P}(\sE)}(1)\otimes\rho^*\cN^{-1}$ satisfies $$
\cM^{q+1}=\cO_{\mathbb{P}(\sE)}(q+1)\otimes \rho^*\omega^{-1}_{C} \simeq \cO_{\mathbb{P}(\sE)}(\Sigma_1+\Sigma_2).$$
It is well known the data $(\cM,\Sigma_1+\Sigma_2)$ gives a finite flat $(q+1)$-cyclic cover $$\pi: S=\mathrm{Spec}_{\cO_{\mathbb{P}(\sE)}}(\cO_{\mathbb{P}(\sE)} \oplus \cM^{-1} \oplus \cdots \oplus \cM^{-q})\to \mathbb{P}(\sE)$$ branched along $\Sigma:=\Sigma_1+\Sigma_2$.
Here the $\cO_{\mathbb{P}(\sE)}$-algebra structure of
$$\cO_{\mathbb{P}(\sE)}\oplus \cM^{-1} \oplus \cdots \oplus \cM^{-q}$$
 is  defined by the embedding $\mathcal{M}^{-q-1} \simeq \mathcal{O}_X(-\Sigma_1-\Sigma_2) \hookrightarrow \cO_X$.

$$\xymatrix{
S\ar[rrrr]^\pi \ar[rrd]_f &&&& \mathbb{P}(\sE) \ar[lld]^\rho\\
&&C
}$$

Since  $\Sigma$ is smooth, the surface $S$ is also smooth.
\begin{rmk}
 In the above construction, we can also replace the $(q+1)$-cyclic cover by  an $s$-cyclic cover with $s\mid (q+1)$ such that the branched divisor remains $\Sigma$. Raynaud's surfaces given in \cite{Ra78} are exactly the above surfaces with $n=1$ and $s=2$ if $p\neq 2$ and $s=3$ if $p=2$.
\end{rmk}

Since $\pi$ is branched along $\Sigma_i, i=1,2$, we have divisors $\Gamma_i\subseteq S$ such that $\pi^*\Sigma_i=(q+1)\Gamma_i$. In particular, the divisor $\Gamma_1$ is again nef and big (cf. \S~\ref{subsec: ruled surface}) and hence $\Gamma_1+f^*R$ is ample for any ample divisor $R$ on $C$.

Finally, we have
\begin{equation}\label{equ: KS}
 \begin{split}
\omega_{S}& \simeq \pi^*(\omega_{\mathbb{P}(\sE)}\otimes \cM^{q})\simeq \pi^*(\cO_{\mathbb{P}(\sE)}(q-2)\otimes \rho^*(\cL\otimes \mathcal{N}))\\
          &\simeq\pi^*(\cO_{\mathbb{P}(\sE)}(q-2)\otimes \rho^*(\cL^2(-l\infty)))
\end{split}
\end{equation}
$$$$
where $$l:=\dfrac{e(qe-3)}{(q+1)}$$ and the third ``$\simeq$''  is due to $\cL\simeq \cO_C(e(qe-3)\cdot\infty)$ and
\begin{equation}\label{equ: equation for l}
\cL^{-1}\otimes \cN\simeq \cO_C((-\dfrac{qe(qe-3)}{q}+\dfrac{qe(qe-3)}{q+1})\cdot \infty )=\cO_C(-l\cdot \infty).
\end{equation}

\smallskip

\section{Adjoint linear systems on generalized Raynaud's surfaces}\label{Sec: adjoint linear systems}
We keep the notations in the previous section. For any ample divisor $R$ on $C$, we set $A_R=\Gamma_1+f^*R$. As   aforementioned, this divisor is ample. The main result is
\begin{thm}\label{thm: main1}
Assume in the defining equation~(\ref{defining equation}) of $C$, the free integral variable $e$ satisfies  \begin{enumerate}[($\bigstar$)]
\item $\,\,\,\,\,\,\,\,\,\,\,\,\,\,\,\,\,\,\,\,\,\,\,\,\,\,\,\,\,\,\,\,\,\,l- (q-1)=\dfrac{e(qe-3)}{q+1}-q + 1\ge e(q-2)$.
\end{enumerate} Then for any $1\le m\le q$ there is an open dense subset $\mathcal{U}_m\subseteq C$ such that for any $Q\in \mathcal{U}_m$ and any Cartier divisor $R$ on $C$ of degree one  satisfying $mR\sim (m-1)\cdot \infty +Q$, the adjoint systems $|K_S+mA_{R}|$ has a base point $\Gamma_2\cap f^{-1}(Q)$.
\end{thm}

The condition $(\bigstar)$ is easily satisfied by taking $e\gg q$, and the existence of $R$ for every $Q$ is guaranteed by the divisibility of $\mathrm{Pic}^0(C)$.  We will see that the condition $(\bigstar)$ is just needed to guarantee a certain element to be zero in $\check{H}^1(\mathcal{U}, \omega_{C})$ when we prove the non-surjectivity of $\phi$ in Lemma~\ref{non-surjectivity}.
The rest part of this section is devoted to proving Theorem~\ref{thm: main1}.

\subsection{A non-free criterion}
From the construction of $S$, we have $$\pi_*\cO_S=\cO_{\mathbb{P}(\sE)}\bigoplus \cM^{-1} \bigoplus \cdots \bigoplus \cM^{-q}.$$
Moreover, we have the following proposition on the decomposition of certain ideal sheaves.
\begin{prop}[cf. \protect{\cite[Prop.~3.3]{Zhe17}}]\label{prop:zheng} For $1\leq r<q+1$ the ideal sheaf $\cO_S(-r\Gamma_1)$ has a decomposition
$$\pi_{*} \mathcal{O}_S(-r \Gamma_1)=(\bigoplus\limits_{i=0}^{r-1} \cM^{-i}(-\Sigma_1))\bigoplus (\bigoplus\limits_{i=r}^{q} \cM^{-i})\subseteq \pi_*\cO_S=\bigoplus\limits_{i=0}^{q} \cM^{-i}$$
which is compatible with the decomposition  of $\pi_{*} \mathcal{O}_{S}$.
\end{prop}
\begin{cor}\label{cor: later use}
Let $\mathcal{F}$ be an invertible sheaf on $\mathbb{P}(\sE)$. Then for $1\leq r<q+1$ we have a decomposition
$$\pi_*(\pi^*\mathcal{F}(-r\Gamma_1))=(\bigoplus\limits_{i=0}^{r-1} \cM^{-i}\otimes \mathcal{F}(-\Sigma_1))\bigoplus (\bigoplus\limits_{i=r}^{q} \cM^{-i}\otimes \mathcal{F}) \subseteq  \bigoplus\limits_{i=0}^{q} \cM^{-i}\otimes \mathcal{F},$$
and the global sections of $\pi^*\mathcal{F}(-r\Gamma_1)$ from all but the first direct summand $\cM^0\otimes \mathcal F(-\Sigma_1)\simeq \mathcal{F}(-\Sigma_1)$ vanish along $\Gamma_2$.

In particular, if the first summand $\mathcal{F}(-\Sigma_1)$ as $\cO_{\mathbb{P}(\sE)}$-module  has a base point $Q_0\in \Sigma_2$, then $\pi^*\mathcal{F}(-r\Gamma_1)$ has a base point $\pi^{-1}(Q_0)\in \Gamma_2$.
\end{cor}
\begin{proof}
The decomposition follows immediately from Proposition \ref{prop:zheng} by the projection formula.
For simplicity we use $\widetilde{\mathcal{F}}_i, i=0,1,\cdots,q$ to denote the corresponding direct summand of $\pi_*(\pi^*\mathcal{F}(-r\Gamma_1))$ in the above decomposition. Then we have
$$H^0(S, \pi^*\mathcal{F}(-r\Gamma_1)) \simeq H^0(\mathbb{P}(\sE), \widetilde{\mathcal{F}}_0) \bigoplus H^0(\mathbb{P}(\sE), \bigoplus_{i>0} \widetilde{\mathcal{F}}_i).$$
For the remaining assertions, it suffices to show the sections from $H^0(\mathbb{P}(\sE), \bigoplus_{i>0} \widetilde{\mathcal{F}}_i)$ do not generate $\pi^*\mathcal{F}(-r\Gamma_1)$ at each point of $\Gamma_2$, or equivalently the natural homomorphism
$$H^0(\mathbb{P}(\sE), \bigoplus_{i>0} \widetilde{\mathcal{F}}_i) \otimes \pi_*\cO_S \to \pi_*(\pi^*\mathcal{F}(-r\Gamma_1))$$
is not surjective at each point of $\Sigma_2 = \pi(\Gamma_2)$.
In practice, we shall verify that the composition homomorphism
$$ \left(\bigoplus_{i=0}^q \cM^{-i}\right) \bigotimes \left(\bigoplus_{i=1}^q \widetilde{\mathcal{F}}_i\right)\to \pi_*(\pi^*\mathcal{F}(-r\Gamma_1)) \to \widetilde{\mathcal{F}}_0= \mathcal{F}(-\Sigma_1)$$
is not surjective at each point of $\Sigma_2$, where the second map denotes the projection to the first direct summand.
Keep in mind the compatibility of this decomposition of $\pi_*(\pi^*\mathcal{F}(-r\Gamma_1))$ as a $\pi_*\cO_S \cong \bigoplus\limits_{i=0}^{q} \cM^{-i}$-module. Locally we can get sections of $\widetilde{\mathcal{F}}_0= \mathcal{F}(-\Sigma_1)$ from the direct summands $\widetilde{\mathcal{F}}_i, i>0$ as follows:
\begin{itemize}
  \item when $1\leq i\leq r-1$, $\cM^{i-q-1}\otimes(\cM^{-i}\otimes \mathcal{F}(-\Sigma_1))\simeq\mathcal{F}(-2\Sigma_1-\Sigma_2)\subset \mathcal{F}(-\Sigma_1)$ as a sub-sheaf determined by tensor with the ideal sheaf $\mathcal {O}(-\Sigma_1-\Sigma_2)$;
  \item when $r\leq i\leq q$, $\cM^{i-q-1}\otimes(\cM^{-i}\otimes \mathcal{F})\simeq \mathcal{F}(-\Sigma_1-\Sigma_2)\subset \mathcal{F}(-\Sigma_1)$ as a sub-sheaf determined by tensor with the ideal sheaf $\mathcal {O}(-\Sigma_2)$.
\end{itemize}
We see that the local sections of $\widetilde{\mathcal{F}}_0$ from $\left(\bigoplus_{i=0}^q \cM^{-i}\right) \otimes(\bigoplus_{i=1}^q \widetilde{\mathcal{F}}_i) $ vanish along $\Sigma_2$ and conclude the result.
\end{proof}

Note that by (\ref{equ: KS}),  we have
\begin{equation}
\begin{split}
\cO_S(K_S+mA_{R})&=\pi^*(\cO_{\mathbb{P}(\sE)}(q-2)\otimes \rho^*(\cL\otimes\mathcal{N}(mR)))(m\Gamma_1)\\
               &=\pi^*(\cO_{\mathbb{P}(\sE)}(q-1)\otimes \rho^*(\cL\otimes\mathcal{N}(mR)))(-(q+1-m)\Gamma_1).
\end{split}
\end{equation}
Therefore, by Corollary~\ref{cor: later use},  Theorem~\ref{thm: main1} follows from the next proposition.
\begin{prop}\label{Prop: main}
Let $\mathcal{P}:=\cO_{\mathbb{P}(\sE)}(q-2)\otimes \rho^*(\cL\otimes \cN)$. When $(\bigstar)$ holds, there is an open dense subset $\mathcal{U}_m\subseteq C$ such that for every closed point $Q\in \mathcal{U}_m$,
the natural map
\begin{equation}
H^0(\mathbb{P}(\sE),\mathcal{P}\otimes \rho^*\cO_C((m-1)\cdot \infty))\stackrel{\otimes \rho^*s_Q}{\to}H^0(\mathbb{P}(\sE),\mathcal{P}\otimes \rho^*\cO_C((m-1)\cdot \infty+Q))
\end{equation}
is an isomorphism, where $s_Q\in H^0(C,\cO_C(Q))$ is a  section corresponding to $Q$.
\end{prop}
In fact, this proposition says that  the fixed part of  the linear system
\begin{align*}
&|\cO_{\mathbb{P}(\sE)}(q-1)\otimes \rho^*(\cL\otimes\mathcal{N}(mR))(-\Sigma_1)|\\
=&|\cO_{\mathbb{P}(\sE)}(q-2)\otimes \rho^*(\cL\otimes\mathcal{N}(mR))|\\
=&|\mathcal{P}\otimes \rho^*\cO_C((m-1)\cdot \infty+Q)|
\end{align*}
contains $F_Q:=\rho^{-1}(Q)$ as a component (and hence admits the base point $F_Q\cap \Sigma_2$).

\subsection{Proof of Proposition~\ref{Prop: main}}
For simplicity, we set $$N:=\left(\dfrac{e(qe-3)}{q+1}-(m-1)\right)=(l+1-m)$$ and $\Delta:=N\cdot \infty$, and hence $\cL\otimes \cN((m-1)\cdot \infty)=\cL^2(-\Delta)$ by (\ref{equ: equation for l}).
An immediate consequence of the assumption $(\bigstar)$ is that
\begin{equation}\label{eq lower bound}
\deg \Delta=N\ge l+1-q \ge e(q-2).
\end{equation}

For Proposition~\ref{Prop: main}, it is equivalent to the following:
\begin{prop}\label{equ: key}
When $(\bigstar)$ holds, there is an open dense subset $\mathcal{U}_m\subseteq C$ such that for every closed point $Q\in \mathcal{U}_m$, the natural map
\begin{equation}
H^0(C, S^{q-2}(\sE)\otimes \cL^2(-\Delta))\stackrel{\otimes s_Q}{\to} H^0(C, S^{q-2}(\sE)\otimes \cL^2(-\Delta+Q))
\end{equation}
is an isomorphism.
\end{prop}

Note that we have the following filtration due to the exact sequence (\ref{ext: ext of E}),
\begin{equation}\label{Filtration of S^q}
S^{-1}(\sE):=0\subseteq S^0(\sE) \subseteq \cdots \subseteq S^{q-1}(\sE).
\end{equation}
In fact, the sheaf  $\sE$ is an extension of $\cL$ by $\cO_C$. So we have the natural ascending filtrations $S^r(\sE)\simeq S^r(\sE)\otimes \cO_C^{q-1-r}\subseteq S^{q-1}(\sE), r=-1,\cdots,q-1$ (see \cite[II, Ex.~5.16]{Har77}). Concretely, with the notation from \S \ref{sec:vectorbundle} the inclusion $S^{r-1}(\sE) \subset S^r(\sE)$ is induced by
$$X_i^{r-1-j}Y_i^j \mapsto X_i^{r-j}Y_i^j, j=0, \cdots r-1, ~\mathrm{on}~U_i.$$

One can also realize the above filtration equivalently in another way as follows. The inclusion $\sE\subseteq  F^n_*\mathcal{O}_C$ along with the $\cO_C$-algebra structure of $F^n_*\cO_C$ induces natural embeddings $S^r(\sE)\subseteq F_*^n\cO_C$ for $r\le q-1$. Concretely, by the construction of $\mathscr{E}$ the embedding can be described as follows:
\begin{equation}\label{Concrete Filtration of S^q}
S^r(\sE)|_{U_i}=\bigoplus\limits_{j=0}^r \cO_C\cdot X_i^{r-j}Y_i^j(=1^{r-j}\cdot \sqrt[q]{z_i^j})\subseteq F^n_*\cO_C|_{U_i}=\bigoplus\limits_{j=0}^{q-1}\cO_C\cdot \sqrt[q]{z_i^j},
\end{equation}
hence the filtration (\ref{Filtration of S^q}) coincides with the natural filtration of $F_*^n\cO_C$.

In either viewpoint, we have canonical isomorphisms $$S^r(\sE)/S^{r-1}(\sE)\simeq \cL^r, r=0,\cdots, q-1.$$  These isomorphisms, in terms of the second viewpoint, are locally given as
$Y_i^r (=\sqrt[q]{z_i^r})\mapsto \eta_i^r$ on $U_i,i=1,2$. We also consider the quotient sheaf $S^{q-2}(\sE)/S^{r-1}(\sE)$, which on $U_i$ is generated by $X_i^{q-2-j}Y_i^{j}=\sqrt[q]{z_i^j}, j=r,\cdots, q-2$.
Then from the filtration (\ref{Filtration of S^q}), we get the exact sequence:
\begin{equation}\label{ext: induction1}
0\to \cL^{r}(\simeq S^{r}(\sE)/S^{r-1}(\sE))\to  S^{q-2}(\sE)/S^{r-1}(\sE)\to S^{q-2}(\sE)/S^{r}(\sE)\to 0.
\end{equation}

\smallskip

\begin{lem}\label{lem: key}
Put $\mathcal{S}^r:=(S^{q-2}(\sE)/S^{r-1}(\sE))\otimes \cL^2, r=0,\cdots, q-2.$ Then for any $r=0,\cdots,q-2$, there is an open dense subset $\mathcal{V}_r\subseteq U_1=C\backslash \infty$ such that the natural maps
$$H^0(C,\mathcal{S}^r(-\Delta)) \stackrel{\otimes s_Q}\to H^0(C, \mathcal{S}^r(-\Delta+Q)) $$
are isomorphisms for all $Q\in \mathcal{V}_r$.
\end{lem}
Note that the assertion in this lemma when $r=0$ is all what we need (see (Proposition~\ref{equ: key})). We are going to prove this lemma by a descending induction on $r$. And before running the induction process, we give the following lemma.
\begin{lem}\label{lem:fundamental}
For any $0\le i\le q$, the natural map
$$H^0(C,\cL^i(-\Delta))\stackrel{\otimes s_Q}{\to} H^0(C,\cL^i(-\Delta+Q))$$ is an isomorphism for any $Q\neq \infty$.
\end{lem}
\begin{proof}
Recall that $\cL^{q} \sim K_C$. By Riemann-Roch formula and Serre duality, it suffices to prove the natural map
$$ H^1(C,\cL^i(-\Delta+Q))^{\vee} \cong H^0(C, \cL^{q-i}(\Delta-Q))\stackrel{\otimes s_Q}{\to} H^0(C, \cL^{q-i}(\Delta)) \cong H^1(C,\cL^i(-\Delta))^{\vee}$$
is not an isomorphism for all $Q\neq \infty$. This is true because $Q$ is not a base point of $|\cL^{q-i}(\Delta)|=| ((q-i)e(qe-3)+N)\cdot \infty|$.
\end{proof}
\begin{proof}[Proof of Lemma~\ref{lem: key}]
First set $r=q-2$, then we have $\mathcal{S}^{q-2}=\cL^{q}$. So Lemma~\ref{lem:fundamental} with $i=q$ gives Lemma~\ref{lem: key} for $r=q-2$ and $\mathcal{V}_{q-2}$ can be taken as $U_1=C\backslash \infty$.

We argue by a descending induction on $r$. Now assume the statement holds for $r+1\leq q-2$. Let's consider the case $r$.
 By tensoring (\ref{ext: induction1}) with $\cL^2$ we obtain the exact sequence
\begin{equation}\label{ext: induction}
0\to \cL^{r+2}\to  \mathcal{S}^r\to \mathcal{S}^{r+1}\to 0.
\end{equation}

Then tensoring the above exact sequence with $\cO_C(-\Delta)$ and $\cO_C(-\Delta+Q)$ for $Q\in \mathcal{V}_{r+1}$, we obtain the following commutative diagram of exact sequences:
{\begin{equation*}
\xymatrix{
0\to \cL^{r+2}(-\Delta) \ar[d]\ar[r] & \mathcal{S}^{r}(-\Delta) \ar[d]\ar[r] &\mathcal{S}^{r+1}(-\Delta)\to 0 \ar[d]\\
0\to \cL^{r+2}(-\Delta+Q)     \ar[r] & \mathcal{S}^{r}(-\Delta+Q)\ar[r]      &\mathcal{S}^{r+1}(-\Delta+Q)\to 0.
}
\end{equation*}}
Taking the cohomology of the above diagram we obtain$$
\Small{\xymatrix{H^0(\cL^{r+2}(-\Delta))\ar[r]\ar[d]^{\otimes s_Q}_{\cong} & H^0(\mathcal{S}^{r}(-\Delta)) \ar[r]\ar[d]^{\alpha_r} &H^0(\mathcal{S}^{r+1}(-\Delta))\ar[r]^{\phi}\ar[d]^{\alpha_{r+1}}_{\cong} & H^1(\cL^{r+2}(-\Delta)) \ar[d]^{\otimes s_Q}\\
H^0(\cL^{r+2}(-\Delta+Q))\ar[r]     &H^0(\mathcal{S}^{r}(-\Delta+Q))\ar[r]   &H^0(\mathcal{S}^{r+1}(-\Delta+Q))\ar[r]^{\phi_Q} &H^1(\cL^{r+2}(-\Delta+Q))\\
}}$$
where the horizontal sequences are exact, the leftmost vertical isomorphism follows from Lemma~\ref{lem:fundamental} and $\alpha_{r+1}$ is an isomorphism by the inductive assumption on $r+1$.

Now we want to find some open dense subset $\mathcal{V}_r\subseteq \mathcal{V}_{r+1}$ such that $\alpha_r$ is an isomorphism for any $Q\in\mathcal{V}_r$. In fact, we only need to verify that $\alpha_r$ is surjective. Applying the Five Lemma, it suffices to verify that the map $\mathrm{Im}(\phi) \xrightarrow{\otimes s_Q} H^1(C,\cL^{r+2}(-\Delta+Q))$ is injective, which is equivalent to the following condition
\begin{itemize}
\item
the intersection of $W:=\mathrm{Im}(\phi)\subseteq H^1(C, \cL^{r+2}(-\Delta))$ with the kernel of the map $ \iota_Q: H^1(C,\cL^{r+2}(-\Delta)) \xrightarrow{\otimes s_Q} H^1(C,\cL^{r+2}(-\Delta+Q))$ is zero, for any  $Q\in \mathcal{V}_r\subseteq \mathcal{V}_{r+1}$.
\end{itemize}

We claim that   \begin{itemize}
                 \item if $W\subset H^1(C, \cL^{r+2}(-\Delta))$ is a proper subspace, then there exists an open dense subset $\mathcal{V}_r\subseteq \mathcal{V}_{r+1}$ such that $\mathrm{Ker}(\iota_Q)\cap W=0$ for every point $Q \in \mathcal{V}_{r}$.
               \end{itemize}

 In fact, if $W \subseteq H^1(C, \cL^{ r+2}(-\Delta))$ is a proper subspace, by Serre duality we can take a nonzero section
$$0\neq \omega_0 \in H^0(C,\cL^{q-r-2}(\Delta)) \cong H^1(C,\cL^{r+2}(-\Delta))^\vee$$
such that the pairing
$$<\omega_0, w>=0,\hspace{.5cm}\forall \, \, w \in W.$$
Let $Z(\omega_0)$ be the set of zeros of $\omega_0$. We take
$$\mathcal{V}_r = (C \setminus Z(\omega_0)) \cap \mathcal{V}_{r+1}.$$
Then for any fixed point $Q \in \mathcal{V}_r$, $\omega_0$ is not contained in the linear subspace $H^0(\cL^{q-r-2}(\Delta-Q))\otimes s_Q \subseteq H^0(\cL^{q-r-2}(\Delta))$, and also note that this subspace is of codimension one, so we have
\begin{align}\label{eq-span}H^0(\cL^{q-r-2}(\Delta)) = (H^0(\cL^{q-r-2}(\Delta - Q))\otimes s_Q) \bigoplus k\cdot \omega_0.\end{align}
 Then for any $s\in \mathrm{ker }(\iota_Q) \cap W$, by definition we can check that
\begin{itemize}
  \item $<\omega' \otimes s_Q, s> =<\omega', s \otimes s_Q> =0$, for any $ \omega' \in H^0(\cL^{q-r-2}(\Delta - Q)) \cong H^1(\cL^{r+2}(-\Delta+Q))^\vee$ and
  \item $<\omega_0, s> =0$.
\end{itemize}
So by (\ref{eq-span}), we see that $<\omega'', s> =0$ for any $ \omega'' \in H^0(\cL^{q-r-2}(\Delta)) $, which means $s=0$.

 In conclusion, to finish the proof, we only need to show  $W\subseteq H^1(C, \cL^{r+2}(-\Delta))$ is a proper subspace. This is done by the next lemma.

\begin{lem}\label{non-surjectivity}
The associated map $$\phi:H^0(C, \mathcal{S}^{r+1}(-\Delta))  \to H^1(C, \cL^{r+2}(-\Delta))$$ is not surjective for $r=0,\cdots, q-3$.
\end{lem}
\begin{proof}
By Serre's duality, we only need to find
$$0\neq \omega \in H^0(C,\omega_{C}\otimes \cL^{-r-2}(\Delta))\cong H^1(C, \cL^{r+2}(-\Delta))^\vee$$
such that
for any $s \in H^0(C, \mathcal{S}^{r+1}(-\Delta))$ the pairing $<\omega,\phi(s)>=0$, which  means
$$\omega \otimes \phi(s) = 0\in H^1(C, \omega_{C}) \cong \check{H}^1(\mathcal{U}=\{U_1,U_2\}, \omega_{C}).$$

We will actually take
$$\omega=dz_1\otimes \eta_1^{-r-2} = \alpha^{q-r-2}\mathrm{d}z_2 \otimes \eta_2^{-r-2} \in H^0(C,\omega_{C}\otimes \cL^{-r-2})\subseteq H^0(C,\omega_{C}\otimes \cL^{-r-2}(\Delta)).$$
Recall that $\eta_i$ are the local generators of $\cL$ on $U_i$ with $\eta_1 = \alpha \cdot \eta_2$ and $\alpha=y^{e(qe-3)}$. Recall also $\beta=y^{-e}$ and $y^N$ is a generator for $\cO_C(-\Delta)$ on $U_2$.
Now take an arbitrary $s \in H^0(C, \mathcal{S}^{r+1}(-\Delta))$.  Note that $\mathcal{S}^{r+1}(-\Delta)$ has  bases $\{\sqrt[q]{z_1^{r+1}}\otimes \eta_1^{2}, \cdots , \sqrt[q]{z_1^{q-2}}\otimes \eta_1^{2}\}$ on $U_1$ and $\{y^N\sqrt[q]{z_2^{r+1}}\otimes \eta_2^{2}, \cdots , y^N\sqrt[q]{z_2^{q-2}}\otimes \eta_2^{2}\}$ on $U_2$. We can write $s$ locally as
\begin{equation}\label{eq:descr s}
\begin{split}
s&=(u_{r+1}\sqrt[q]{z_1^{r+1}}+\cdots +u_{q-2} \sqrt[q]{z_1^{q-2}})\otimes \eta_1^{2}\\
 &=y^N(v_{r+1}\sqrt[q]{z_2^{r+1}}+\cdots +v_{q-2} \sqrt[q]{z_2^{q-2}})\otimes \eta_2^{2}.
\end{split}
\end{equation}
with $u_i\in \mathcal {O}_{C}(U_1)$ and $v_i\in \mathcal {O}_{C}(U_2)$.

To calculate the element $\phi(s) \in \check{H}^1(\mathcal{U}, \cL^{r+2}(-\Delta))$, we lift
$s$ to sections of $\mathcal{S}^{r}(-\Delta)$ locally on $U_i$  as follows
$$s_1=(u_{r+1}\sqrt[q]{z_1^{r+1}}+\cdots + u_{q-2}\sqrt[q]{z_1^{q-2}})\otimes \eta_1^{2}~\mathrm{on}~U_1$$ and
$$s_2=(v_{r+1}\sqrt[q]{z_2^{r+1}}+\cdots +v_{q-2}\sqrt[q]{z_2^{q-2}})\otimes y^N\eta_2^{2}~\mathrm{on}~U_2.$$
Then the desired element $\phi(s)\in \check{H}^1(\mathcal{U}, \cL^{r+2}(-N \cdot \infty))$ is represented by $s_2-s_1\in H^0(U_1\cap U_2, \cL^{r+2}(-\Delta))$. To describe it, we recall the relation $z_1=\alpha^qz_2+\beta^q$ from (\ref{equ: translation relation}), and with the notation (\ref{Concrete Filtration of S^q}), we may regard $\sqrt[q]{z_i^r}= \eta_i^r$ as the generator of $\cL^{r} \simeq S^r(\sE)/S^{r-1}(\sE)$ on $U_i$.
Since $(s_2-s_1)|_{U_1 \cap U_2} \equiv 0 ~\mathrm{mod} ~\cL^{r+2}(-\Delta)$, if we replace $\sqrt[q]{z_2}$ by $\alpha^{-1}(\sqrt[q]{z_1}-\beta)$ in $s_2$, we obtain that
$$(s_2-s_1)|_{U_1 \cap U_2}\equiv 0 ~\mathrm{mod} ~\mathcal {O}_{C}\cdot\sqrt[q]{z_1^r}\otimes y^N\eta_2^{2}.$$
Therefore, by the relation $\sqrt[q]{z_1}=\eta_1= \alpha\cdot \eta_2$ of local generators of $\cL$ we may write that
$$(s_2-s_1)|_{U_1 \cap U_2}=\mu \cdot \sqrt[q]{z_1^r}\otimes y^N \eta_2^{2} =\mu y^N \cdot\alpha^r \otimes \eta_2^{r+2}\in H^0(U_1\cap U_2,\cL^{r+2}(-\Delta))$$
where $\mu$ denote the coefficient of $\sqrt[q]{z_1^r}$ in the polynomial expansion of
$$v_{r+1}(\frac{\sqrt[q]{z_1}-\beta}{\alpha})^{r+1}+\cdots +v_{q-2}(\frac{\sqrt[q]{z_1}-\beta}{\alpha})^{q-2}.$$
More precisely
\begin{align*}
 \mu=\alpha^{-r}(C_{r+1}^rv_{r+1} (\dfrac{-\beta}{\alpha}) + \cdots+C_{q-2}^rv_{q-2}(\dfrac{-\beta}{\alpha})^{q-2-r}).
\end{align*}

In turn, the element $\omega \otimes \phi(s)\in \check{H}^1(\mathcal{U}, \omega_{C})$ is represented by
\begin{align*}
& y^N\alpha^{q-r-2}(C_{r+1}^rv_{r+1} (\dfrac{-\beta}{\alpha}) + \cdots+C_{q-2}^rv_{q-2}(\dfrac{-\beta}{\alpha})^{q-2-r})\cdot \mathrm{d}z_2\\
  =&y^N(C^r_{r+1}v_{r+1}\cdot \alpha^{q-r-3} (-\beta)+\cdots+C^r_{q-2}v_{q-2}\cdot (-\beta)^{q-r-2})\mathrm{d}z_2
\end{align*}in $H^0(U_1\cap U_2,\omega_{C})$.
Since $v_i$ are regular on $U_2$, $\alpha=y^{e(qe-3)}$  and $\beta=y^{-e}$, we have
\begin{align*}
&v_\infty(y^N(C^r_{r+1}v_{r+1}\cdot \alpha^{q-r-3} (-\beta)+\cdots+C^r_{q-2}v_{q-2}\cdot (-\beta)^{q-r-2}))\\
\ge&v_\infty(y^N\cdot \beta^{q-r-2})=N-e(q-r-2)\ge N-e(q-2) \geq 0.
\end{align*}
Hence $\omega\otimes \phi(s)$ extends regularly on $U_2$, which indicates $\omega \otimes \phi(s) = 0$ in $\check{H}^1(\mathcal{U}, \omega_{C})$.
\end{proof}

\end{proof}

\subsection{Conclusion}
In summary, for any fixed algebraically closed field $\sk$ of characteristic $p>0$ and any positive integer $m\ge 1$, we can take $q=p^n\ge m$ and $e=(q+1)e_0$ for some $e_0\in \mathbb{N}_+$ such that $(\blacksquare)$ and $(\bigstar)$ hold. Then Theorem~\ref{thm: main1} shows that on the associated smooth projective surface $S$ constructed in \S~\ref{Sec: a generalization of Raynaud's surface}, there is an ample divisor $A$ such that $|K_S+mA|$ is not free of base points. Namely, we have proven Theorem~\ref{thm-main}.
\smallskip

\bibliographystyle{plain}

\end{document}